\documentclass[12pt,a4paper]{article}
\usepackage{amsmath}
\usepackage{amssymb}
\usepackage{amssymb}
\usepackage{cases}
\usepackage{bbm}
\usepackage{mathrsfs}
\usepackage{graphicx}
\usepackage{amsfonts}
\usepackage{enumerate}
\usepackage{theorem}

\makeatletter
\def\tank#1{\protected@xdef\@thanks{\@thanks
 \protect\footnotetext[0]{#1}}}
\def\bigfoot{

 \@footnotetext}
\makeatother

\newcommand{\ea}{\end{array}}

\allowdisplaybreaks

\newtheorem{theorem}{Theorem}[section]
\newtheorem{lem}{Lemma}[section]
\newtheorem{prop}[theorem]{Proposition}
\newtheorem{thm}[theorem]{Theorem}
\newtheorem{cor}[theorem]{Corollary}

\newtheorem{rem}{Remark}[section]

\def\beq{\begin{equation}}
\def\nneq{\end{equation}}

\def\bthm{\begin{thm}}
\def\nthm{\end{thm}}

\def\blem{\begin{lem}}
\def\nlem{\end{lem}}
\def\bprf{\begin{proof}}
\def\nprf{\end{proof}}
\def\bprop{\begin{prop}}
\def\nprop{\end{prop}}
\def\brmk{\begin{rem}}
\def\nrmk{\end{rem}}

\def\<{\left<}\def\>{\right>}\def\({\left(}\def\){\right)}

\def\bexa{\begin{exa}}
\def\nexa{\end{exa}}
\def\bcor{\begin{cor}}
\def\ncor{\end{cor}}

\newcommand{\rr}{\mathbb{R}}

\def\e{\varepsilon}

{\theorembodyfont{\rmfamily}
}

\title{A Moderate Deviation Principle for 2-D Stochastic Navier-Stokes Equations Driven by
Multiplicative L\'evy Noises}

\footnotesize{\author{Zhao Dong$^{1,}$\thanks{dzhao@amt.ac.cn. ZD's research is supported by Key Laboratory of Random Complex Structures and Data Science£¬Academy of Mathematics and Systems Science£¬Chinese Academy of  Sciences, by 973 Program (2011CB808000) and by NSFC, No£º11271356 and No£º11371041.}, Jie
Xiong$^{2,}$\thanks{jiexiong@umac.mo. JX's research is supported by Macao Science and Technology Fund FDCT
076/2012/A3 and Multi-Year Research Grants of the University of Macau No.
MYRG2014-00015-FST and MYRG2014-00034-FST.}, Jianliang Zhai
$^{3,}$\thanks{zhaijl@ustc.edu.cn},\ \
 Tusheng Zhang$^{4,3,}$\thanks{Tusheng.Zhang@manchester.ac.uk}\\
 {\em $^1$ Academy of Mathematics and Systems Sciences,}\\
 {\em Chinese Academy of Sciences (CAS),}\\
 {\em Beijing, 100080, China}\\
 {\em $^2$ Department of Mathematics,}\\
{\em Faculty of Science and Technology,}\\
 {\em University of Macau,}\\
 {\em Taipa, Macau, China}\\
 {\em $^3$ School of Mathematical Sciences,}\\
 {\em University of Science and Technology of China,}\\
 {\em Hefei, 230026, China}\\
 {\em $^4$ School of Mathematics, University of Manchester,}\\
 {\em Oxford Road, Manchester, M13 9PL, UK}\\
}
\date{}
\newenvironment{proof}{\par\noindent{\bf Proof:}}{\hspace*{\fill}$\blacksquare$\par}
\begin{document}
\maketitle

\begin{minipage}{140mm}
\begin{center}
{\bf Abstract}
\end{center}
In this paper, we establish a moderate deviation principle for two-dimensional stochastic Navier-Stokes equations driven by multiplicative $L\acute{e}vy$ noises. The weak convergence method introduced by Budhiraja, Dupuis and Ganguly in \cite{Budhiraja-Dupuis-Ganguly}  plays a key role.
\end{minipage}

\vspace{4mm}

\noindent \textbf{AMS Subject Classification}: Primary 60H15 Secondary 35R60, 37L55.

\vspace{3mm}
\noindent \textbf{Key Words:} Moderate deviation principles; Stochastic Navier-Stokes equations; Poisson random measures; Skorokhod representation; Tightness.

\vspace{4mm}

\section{Introduction}
%The small noise large deviation principles (LDPs) of Freidlin-Wentzell type for stochastic (partial) differential equations have been extensively developed, (cf. \cite{Da-Zab}, \cite{Dembo-Zeitouni}, etc.  ). The aim of this paper is
%to study the central limit theorem and  the moderate deviations for $2$-D stochastic  Navier-Stokes equations.

Consider the two-dimensional Navier-Stokes equation with Dirichlet boundary condition, which  describes the time evolution of an incompressible fluid,

\begin{equation}\label{PDE}
 \frac{\partial u(t,x)}{\partial t}-\nu\Delta u(t,x)+(u(t,x)\cdot\nabla)u(t,x)+\nabla p(t,x)
=
  f(t,x),
\end{equation}
with the conditions
\begin{equation}\label{conditions}
    \begin{cases}  (\nabla\cdot u)(t,x)=0, \ \ \  \text{for}\ \ x\in D,\ \  t>0,  \\
    u(t,x)=0,  \ \ \ \ \qquad  \text{for}\ \ x\in \partial D, t\geq0,\\
    u(0,x)=u_0(x),  \ \ \ \ \text{for}\ \ x\in D,
                        \end{cases} \end{equation}
where $D$ is a bounded open domain of $\rr^2$ with regular boundary $\partial D$, $u(t,x)\in\rr^2$ denotes the velocity field at time $t$ and position $x$,
$\nu>0$ is the viscosity coefficient, $p(t,x)$ denotes the pressure field, $f$ is a deterministic external force.

To formulate the Navier-Stokes equation, we introduce the following standard spaces: let
$$
V=\left\{v\in H_0^1(D;\rr^2):\nabla\cdot v=0, \text{a.e. in } D\right\},
$$
with the norm
$$
\|v\|_V:=\left(\int_D|\nabla v|^2dx\right)^{\frac12}=\|v\|,
$$ and let $H$ be the closure of $V$ in the $L^2$-norm
$$
|v|_H:=\left(\int_D|v|^2dx\right)^{\frac12}=|v|.
$$

 Define the operator $A$ (Stokes operator) in $H$ by the formula
 $$
 Au:=-  \nu P_H \Delta u,\ \ \ \ \forall u\in H^2(D;\rr^2)\cap V,
 $$
where the linear operator $P_H$ (Helmhotz-Hodge projection) is the projection operator from $L^2(D;\rr^2)$ to $H$,
and define the nonlinear operator $B$ by
$$
B(u,v):=P_H((u\cdot\nabla)v),
$$
with the notation $B(u):=B(u,u)$ for short.

By applying the operator $P_H$ to each term of \eqref{PDE}, we can rewrite it in the following abstract form:
\beq\label{SNSE}
du(t)+Au(t)dt+B(u(t))dt=f(t)dt \ \ \ \text{    in } L^2([0,T],V'),
\nneq
with the initial condition $u(0)=u_0$ for some fixed point $u_0$ in $H$.
\vskip0.3cm

Taking into account the random external forces, in this paper we consider stochastic Navier-Stokes equations (SNSE) driven by the
multiplicative $L\acute{e}vy$ noise, that is, the following random perturbations of Navier-Stokes equation:
\begin{eqnarray}\label{SNSE-01}
\left\{
 \begin{array}{lll}
 & \hbox{ $du^\epsilon(t)
 =
 -Au^\epsilon(t)dt-B(u^\epsilon(t))dt+f(t)dt+\epsilon\int_\mathbb{X}G(u^\epsilon(t-),v)\widetilde{N}^{\epsilon^{-1}}(dtdv)$;} \\
 & \hbox{$\ \ u^\epsilon(0)=u_0\in H$.}
 \end{array}
\right.
\end{eqnarray}
Here $\mathbb{X}$ is a locally compact Polish space, $G$ is a measurable mapping to be specified later, $N^{\epsilon^{-1}}$ is a Poisson random measure on $[0,T]\times\mathbb{X}$ with a
$\sigma$-finite
intensity measure $\epsilon^{-1}\lambda_T\otimes \vartheta$, $\lambda_T$ is the Lebesgue measure on $[0,T]$ and $\vartheta$ is a $\sigma$-finite measure on $\mathbb{X}$, $\widetilde{N}^{\epsilon^{-1}}$ is the compensated Poisson random measure, i.e., for $O\in \mathcal{B}(\mathbb{X})$ with $\vartheta(O)<\infty$,
\[\widetilde{N}^{\epsilon^{-1}}([0,t]\times O)=N^{\epsilon^{-1}}([0,t]\times O)-\epsilon^{-1}t\vartheta(O).\]

\vskip 0.3cm

As the parameter $\e$ tends to zero, the solution $u^\e$ of
\eqref{SNSE-01} will tend to the solution of the following
deterministic Navier-Stokes equation at least in the mean sense
\beq\label{NSE} du^0(t)+Au^0(t)dt+B(u^0(t))dt=f(t)dt,\ \ \
\text{with}\ u^0(0)=u_0\in H. \nneq \vskip 0.3cm
  In this paper, we shall investigate deviations of $u^\e $ from the deterministic solution $u^0$, as $\e$ decreases to $0$,
that is,  the asymptotic behavior of the trajectory,
\begin{eqnarray}\label{eq Y}
Y^\e=\left(u^\e -u^0\right)/{a(\e)},
\end{eqnarray}
where $a(\e)$ is some deviation scale which strongly influences the asymptotic behavior of $Y^\e$.
We will study the so-called moderate deviation principle (MDP for short, cf. \cite{Dembo-Zeitouni}), that is when the deviation scale satisfies
\beq \label{h}
 a(\e)\to0,\ \ \e/{a^2(\e)}\to0\ \ \ \text{as}\ \e\to0.
 \nneq
Throughout this paper, we assume that \eqref{h} is in place.

 \vskip 0.3cm
Large deviations for  stochastic partial differential equations   have been
investigated in many papers, see \cite{CW}, \cite{CR},  \cite{KX}, \cite{S}, etc..
%Since the work of Bensoussan and Temam \cite{Ben-Temam}, stochastic Navier-Stokes equations have been intensively studied, see \cite{Da-Zab}
% for the equation with additive Gaussian noise. The existence and uniqueness of solutions for the 2-D stochastic Navier-Stokes equations with multiplicative
%Gaussian noise were obtained in \cite{Flandoli-Gatarek},  \cite{Srith-Sundar}. The ergodic properties and invariant
%measures of the 2-D stochastic Navier-Stokes equations were studied in \cite{Flandoli} and \cite{Hairer-Matt}.
Wentzell-Freidlin type
large deviation results for the two-dimensional stochastic Navier-Stokes equations
with Gaussian noise have been established in \cite{AX} and \cite{Srith-Sundar}, and the case of   L\'evy noise has been established in \cite{Xu-Zhang} and \cite{Zhai-Zhang}.
\vskip 0.3cm

 Like the  large deviations, the moderate deviation problems arise in  the theory of statistical  inference quite naturally.
 The estimates of moderate deviations can provide us with the rate
  of convergence and a useful method for constructing asymptotic confidence intervals, see \cite{Erm}, \cite{Gao-Zhao11}, \cite{IK}, \cite{Kal}  and the references therein.
Results on the MDP for processes with independent increments were obtained in
De Acosta \cite{DeA}, Chen \cite{Ch} and Ledoux \cite{Led}. The study of
the MDP estimates for other processes has been carried out as well, e.g., Wu \cite{Wu} for Markov processes,  Guillin  and Liptser \cite{GL} for diffusion processes,  Wang and Zhang \cite{WZ} for stochastic reaction-diffusion equations. Wang $et\; al$ \cite{WZZ} considered a MDP for 2-D stochastic Navier-Stokes equations driven by multiplicative Wiener processes.

The moderate deviation problems for stochastic evolution equations and stochastic
 partial differential equations driven by $L\acute{e}vy$ noise are drastically different because of the appearance of the jumps.
  There is not much  study on this topic so far.
Recently, Budhiraja $et\; al$ \cite{Budhiraja-Dupuis-Ganguly} obtained the MDPs
for stochastic differential equations driven by a Poisson random measure in finite dimensions and
in some co-nuclear spaces, which can not cover SNSEs.
\vskip 0.4cm

Our aim is to establish a moderate deviation principle for the
two-dimensional stochastic Navier-Stokes equations (SNSEs) driven by
multiplicative $L\acute{e}vy$ noises. We will apply the abstract
criteria (weak convergence approach) obtained in
\cite{Budhiraja-Dupuis-Ganguly}. However, it is quite non-trivial to
implement the weak convergence approach to the SNSEs due to the highly
non-linear term in the equation and the appearance of the jumps. The
crucial step is to show the weak convergence of the SNSEs driven by
counting random measures  with random intensity. To this end, we
decompose the solutions into a sum of the solutions of several
relatively simpler equations and prove the convergence/tightness of
the solutions of each equations.

\vskip0.3cm

 The organization of this paper is as follows. In Section 2,
 we recall the general criteria for a moderate deviation principle given in \cite{Budhiraja-Dupuis-Ganguly}.
 Section 3 is devoted to establishing the moderate deviation principle for the two-dimensional stochastic Navier-Stokes equations driven by multiplicative $L\acute{e}vy$ noises.

\vskip0.3cm
Throughout this paper, $c_N, c_{f, T},\cdots$ are positive constants depending on some parameters $N, f, T, \cdots$, independent of $\e$,
 whose value may be different from line to line.

\section{Preliminaries}

In this section, we will recall the general criteria for a moderate deviation principle given in \cite{Budhiraja-Dupuis-Ganguly}, and to this end, we closely follow the framework and the notations in that paper.

\subsection{Controlled Poisson random measure}\label{Section Representation}

Let $\mathbb{X}$ be a locally compact Polish space. Denote by $\mathcal{M}_{FC}(\mathbb{X})$ the space of all
measures $\vartheta$ on $(\mathbb{X},\mathcal{B}(\mathbb{X}))$ such that $\vartheta(K)<\infty$ for every compact $K$ in
$\mathbb{X}$, and let $C_c(\mathbb{X})$ be
the space of continuous functions with compact supports. Endow $\mathcal{M}_{FC}(\mathbb{X})$ with the weakest topology such that for every $f\in
C_c(\mathbb{X})$, the function
\[\vartheta\rightarrow\< f,\vartheta\>=\int_{\mathbb{X}}f(u)d\vartheta(u)\]
 is continuous   in $\vartheta\in\mathcal{M}_{FC}(\mathbb{X})$.
This topology can be metrized such that $\mathcal{M}_{FC}(\mathbb{X})$ is a Polish space (see e.g.
\cite{Budhiraja-Dupuis-Maroulas.}).
Fix $T\in(0,\infty)$ and let $\mathbb{X}_T=[0,T]\times\mathbb{X}$. Fix a measure
$\vartheta\in\mathcal{M}_{FC}(\mathbb{X})$,
and let $\vartheta_T=\lambda_T\otimes\vartheta$, where $\lambda_T$ is Lebesgue measure on $[0,T]$.
\vskip0.3cm

We recall that a Poisson random measure $\textbf{n}$ on $\mathbb{X}_T$ with intensity measure
$\vartheta_T$ is an $\mathcal{M}_{FC}(\mathbb{X}_T)$ valued random variable such that for each
$B\in\mathcal{B}(\mathbb{X}_T)$
with $\vartheta_T(B)<\infty$, $\textbf{n}(B)$ is Poisson distributed with mean $\vartheta_T(B)$ and for disjoint
$B_1,\cdots,B_k\in\mathcal{B}(\mathbb{X}_T)$, $\textbf{n}(B_1),\cdots,\textbf{n}(B_k)$ are  independent
random
variables (cf. \cite{Ikeda-Watanabe}). Denote by $\mathbb{P}$ the measure induced by $\textbf{n}$ on
$(\mathcal{M}_{FC}(\mathbb{X}_T),\mathcal{B}(\mathcal{M}_{FC}(\mathbb{X}_T)))$.
Then letting $\mathbb{M}=\mathcal{M}_{FC}(\mathbb{X}_T)$, $\mathbb{P}$ is the unique probability measure on
$(\mathbb{M},\mathcal{B}(\mathbb{M}))$
under which the canonical map, $N:\mathbb{M}\rightarrow\mathbb{M},\ N(m)\doteq m$, is a Poisson random measure with
intensity measure $\vartheta_T$. We also consider, for $\theta>0$,
probability
measures $\mathbb{P}_\theta$ on $(\mathbb{M},\mathcal{B}(\mathbb{M}))$ under which $N$ is a Poissson random measure
with intensity $\theta\vartheta_T$. The corresponding expectation operators will be denoted by $\mathbb{E}$ and
$\mathbb{E}_\theta$,
respectively.
\vskip0.3cm

Set $\mathbb{Y}=\mathbb{X}\times[0,\infty)$ and
$\mathbb{Y}_T=[0,T]\times\mathbb{Y}$. Similarly, let
$\bar{\mathbb{M}}=\mathcal{M}_{FC}(\mathbb{Y}_T)$ and let
$\bar{\mathbb{P}}$ be the unique probability measure on
$(\bar{\mathbb{M}},\mathcal{B}(\bar{\mathbb{M}}))$ under which the
canonical map,
$\bar{N}:\bar{\mathbb{M}}\rightarrow\bar{\mathbb{M}},\bar{N}(\bar
m)\doteq \bar m$, is a Poisson random measure with intensity measure
$\bar{\vartheta}_T=\lambda_T\otimes\vartheta\otimes \lambda_\infty$,
with $\lambda_\infty$ being Lebesgue measure on $[0,\infty)$. The
corresponding expectation operator will be denoted by
$\bar{\mathbb{E}}$. Let
\[\mathcal{F}_t\doteq\sigma\{\bar{N}((0,s]\times O):0\leq s\leq t,
O\in\mathcal{B}(\mathbb{Y})\},\] and denote by $\bar{\mathcal{F}}_t$
 the completion of $\mathcal{F}_t$ under $\bar{\mathbb{P}}$. Let $\bar{\mathcal{P}}$ be the predictable $\sigma$-field on
$[0,T]\times\bar{\mathbb{M}}$
with the filtration $\{\bar{\mathcal{F}}_t:0\leq t\leq T\}$ on $(\bar{\mathbb{M}},\mathcal{B}(\bar{\mathbb{M}}))$.
Let
$\bar{\mathcal{A}}_+$ \textbf{(}resp. $\bar{\mathcal{A}}$\textbf{)}
be the class of all $(\bar{\mathcal{P}}\otimes\mathcal{B}(\mathbb{X}))/\mathcal{B}[0,\infty)$\textbf{(}resp. $(\bar{\mathcal{P}}\otimes\mathcal{B}(\mathbb{X}))/\mathcal{B}(\mathbb{R})$\textbf{)}-measurable maps
$\varphi:\mathbb{X}_T\times\bar{\mathbb{M}}\rightarrow[0,\infty)$ \textbf{(}resp. $\varphi:\mathbb{X}_T\times\bar{\mathbb{M}}\rightarrow\mathbb{R}$\textbf{)}. For $\varphi\in\bar{\mathcal{A}}_+$, define a
stochastic counting measure $N^\varphi$ on
$\mathbb{X}_T$ by
 \begin{eqnarray}\label{Jump-representation}
 N^\varphi((0,t]\times U)=\int_{(0,t]\times U}\int_{(0,\infty)}1_{[0,\varphi(s,x)]}(r)\bar{N}(dsdxdr),\
 t\in[0,T],U\in\mathcal{B}(\mathbb{X}).
 \end{eqnarray}
$N^\varphi$ is the controlled random measure, with $\varphi$ selecting the intensity for the points at location
$x$
and time $s$, in a possibly random but non-anticipating way. When $\varphi(s,x,\bar{m})\equiv\theta\in(0,\infty)$, we
write $N^\varphi=N^\theta$. Note that $N^\theta$ has the same distribution with respect to $\bar{\mathbb{P}}$ as $N$
has with respect to $\mathbb{P}_\theta$.
\vskip0.3cm

We end this subsection with some notations.
Define $l:[0,\infty)\rightarrow[0,\infty)$ by
 $$
 l(r)=r\log r-r+1,\ \ r\in[0,\infty).
 $$

For any $\varphi\in\bar{\mathcal{A}}_+$ the quantity
 \begin{eqnarray}\label{L_T}
 L_T(\varphi)=\int_{\mathbb{X}_T}l(\varphi(t,x,\omega))\vartheta_T(dtdx)
 \end{eqnarray}
is well defined as a $[0,\infty]$-valued random variable. Let $\{K_n\subset \mathbb{X},\
n=1,2,\cdots\}$
be an increasing sequence of compact sets such that $\cup _{n=1}^\infty K_n=\mathbb{X}$. For each $n$ let
 \begin{eqnarray*}
 \bar{\mathcal{A}}_{b,n}
 &\doteq&
 \{\varphi\in\bar{\mathcal{A}}_+:
 \ for\ all\ (t,\omega)\in[0,T]\times\bar{\mathbb{M}},\
 n\geq\varphi(t,x,\omega)\geq
 1/n\ if\ x\in K_n\ \\
 &&\ \ \ \ \ \ \ \ \ \ \ \ \ \ \ \ \ \ \ \ \ \ \ \ \ \ \ \ \ \ \ \ \ \ \ \ \ \ \ \ \
 \ \ \ \
 and\ \varphi(t,x,\omega)=1\ if\ x\in K_n^c
 \},
 \end{eqnarray*}
and let $\bar{\mathcal{A}}_{b}=\cup _{n=1}^\infty\bar{\mathcal{A}}_{b,n}$.

\subsection{A General Moderate Deviation Result}
In this subsection, we recall a general criteria for a moderate deviation principle introduced in \cite{Budhiraja-Dupuis-Ganguly}.

Assume that $a(\e)$ satisfies (\ref{h}). Let
$\{\mathcal{G}^\epsilon\}_{\epsilon>0}$
be a family of measurable maps from $\mathbb{M}$ to $\mathbb{U}$, where $\mathbb{M}$ is introduced in
Subsection \ref{Section Representation} and $\mathbb{U}$ is a Polish space.
We present
below a sufficient condition for large deviation principle (LDP in abbreviation) to hold for the family
$\mathcal{G}^\epsilon(\epsilon N^{\epsilon^{-1}})$ as $\epsilon\rightarrow 0$, with speed $\e/{a^2(\e)}$ and a rate function that
is given though a suitable quadratic form, which is the so-called moderate deviation principle (MDP for short, cf. \cite{Dembo-Zeitouni}).

For $\e>0$ and $M<\infty$, consider the spaces
\begin{eqnarray}\label{eq s}
& &\mathcal{S}^M_{+,\e}=\{\varphi:\ \mathbb{X}\times[0,T]\to\mathbb{R}_+\ |\ L_T(\varphi)\leq M a^2(\e)\}\\
& &\mathcal{S}^M_{\e}=\{\psi:\ \mathbb{X}\times[0,T]\to\mathbb{R}\ |\ \psi=(\varphi-1)/{a(\e)},\ \varphi\in\mathcal{S}^M_{+,\e}\}.\nonumber
\end{eqnarray}
We also let
\begin{eqnarray}\label{eq u}
& &\mathcal{U}^M_{+,\e}=\{\varphi\in\bar{\mathcal{A}}_b:\ \varphi(\cdot,\cdot,\omega)\in\mathcal{S}^M_{+,\e},\ \bar{\mathbb{P}}\text{-}a.s.\}\\
& &\mathcal{U}^M_{\e}=\{\psi\in\bar{\mathcal{A}}:\ \psi(\cdot,\cdot,\omega)\in\mathcal{S}^M_{\e},\ \bar{\mathbb{P}}\text{-}a.s.\}\nonumber
\end{eqnarray}

The norm in the Hilbert space $L^2(\vartheta_T)$ will be denoted by
$\|\cdot\|_2$ and $B_2(R)$ denotes the ball of radius $R$ in
$L^2(\vartheta_T)$. Throughout this paper $B_2(R)$ is equipped with
the weak topology of $L^2(\vartheta_T)$ and it is therefore weakly
compact. Given a map $\mathcal{G}_0:\ L^2(\vartheta_T)\to\mathbb{U}$
and $\eta\in\mathbb{U}$, let
$$
\mathbb{S}^0_\eta=\{\psi\in L^2(\vartheta_T):\ \eta=\mathcal{G}_0(\psi)\}
$$
and define $I$ by
\begin{eqnarray}\label{eq I}
I(\eta)=\inf_{\psi\in\mathbb{S}^0_\eta}\Big[\frac{1}{2}\|\psi\|^2_2\Big].
\end{eqnarray}
By convention, $I(\eta)=+\infty$ if $\mathbb{S}^0_\psi=\emptyset$.
\vskip0.3cm

Suppose $\varphi\in\mathbb{S}^M_{+,\e}$. By Lemma 3.2 in \cite{Budhiraja-Dupuis-Ganguly}, there exists
 $\kappa_2(1)\in(0,\infty)$ that is
independent of $\e$ and such that $\psi1_{\{|\psi|\leq 1/{a(\e)}\}}\in B_2(\sqrt{M\kappa_2(1)})$,
where $\psi=(\varphi-1)/{a(\e)}$. In this
paper, we use the symbol $``\Rightarrow"$ to denote convergence in distribution.
\vskip0.3cm

{\bf Condition MDP: } Let $\mathcal{G}_0: L^2(\vartheta_T)\to
\mathbb{U}$ be measurable and satisfy:

\begin{itemize}
  \item[\textbf{(MDP-1)}] Given $M>0$, suppose that $g^\e,\ g\in B_2(M)$ and $g^\e\to g$. Then
  $$
  \mathcal{G}_0(g^\e)\to\mathcal{G}_0(g)\ \ \ \  \rm in\ \mathbb{U}.
  $$
  \item[\textbf{(MDP-2)}] Given $M>0$, let $\{\varphi^\e\}_{\e>0}$ be such that for every $\e>0$,
  $\varphi^\e\in\mathcal{U}^M_{+,\e}$ and for some
  $\beta\in(0,1]$, $\psi^\e1_{\{|\psi^\e|\leq \beta/a(\e)\}}\Rightarrow\psi$ in $B_2(\sqrt{M\kappa_2(1)})$
  where $\psi^\e=(\varphi^\e-1)/a(\e)$.
  Then
  \[
  \mathcal{G}^\e(\e N^{\e^{-1}\varphi^\e})\Rightarrow \mathcal{G}_0(\psi)\ \ \ \ \rm in\ \mathbb{U}.
  \]
\end{itemize}
\vskip0.3cm

The following criteria was established in \cite{Budhiraja-Dupuis-Ganguly}.
\begin{thm}\label{thm MDP}
Suppose that the functionals $\mathcal{G}^\e$ and $\mathcal{G}_0$ satisfy {\rm\textbf{Condition MDP}}. Then $\{Y^\e\equiv  \mathcal{G}^\e(\e N^{\e^{-1}}), \varepsilon>0\}$ satisfies a large deviation principle with speed $\e/a^2(\e)$ and rate function $I$ defined in (\ref{eq I}).
\end{thm}

\section{Moderate Deviation Principles}

Let $V, H$ be the Hilbert spaces introduced in Section 1.
 Denote by $V'$ the dual of $V$.  Identifying $H$ with its dual $H'$, we have the dense, continuous embedding
$$
V\hookrightarrow H\cong H'\hookrightarrow V'.
$$
In this way, we may consider $A$ as a bounded operator from $V$ to
$V'$. The inner product in $H$ is denoted by $\< \cdot,
\cdot\>$. Moreover, we denote by $(\cdot,\cdot)$, the duality
between $V$ and $V'$. Hence, for $u=(u_i)\in V$, $w=(w_i)\in V$, we
have \beq\label{A} ( Au,w)= \nu\sum_{i,j=1}^2\int_D \partial_i
u_j\partial _i w_j dx. \nneq
 Define $b(\cdot,\cdot,\cdot):V\times V\times V\rightarrow \rr $ by
 \beq\label{b}
 b(u,v,w)=\sum_{i,j=1}^2\int_D u_i\partial_iv_jw_jdx.
\nneq In particular, if $u,v,w\in V$, then
$$
( B(u,v),w)=( (u\cdot\nabla)v,w)=\sum_{i,j=1}^2\int_D u_i\partial_iv_jw_jdx=b(u,v,w).
$$
$B(u)$ will be used to denote $B(u,u)$. By integration by parts,
\beq\label{b1}
b(u,v,w)=-b(u,w,v),
\nneq
therefore
\beq\label{b2}
b(u,v,v)=0, \ \ \ \ \forall u,v\in V.
\nneq

The following well-known estimates for $b$ (see \cite{Temam} and \cite{Srith-Sundar} for example) will be required in the rest of this paper:
\begin{align}
%&\label{b3}|b(u,v,w)|\le c\|u\|\cdot\|v\|\cdot\|w\|,\\
&\label{b4}|b(u,v,w)|\le 2\|u\|^{\frac12}\cdot|u|^{\frac12}\cdot\|v\|^{\frac12}\cdot|v|^{\frac12}\cdot\|w\|,\\
&\label{b6}|b(u,u, v)|\le \frac12\|u\|^2+32\|v\|_{L^4}^4\cdot|u|^2,\\
&\label{b7}|( B(u)-B(v),u-v)|\le \frac12\|u-v\|^2+c |u-v|^2\cdot \|v\|^4_{L^4},
\end{align}
where \beq\label{L4}
\|v\|^4_{L^4}\le \|v\|^2|v|^2.
\nneq

\vskip0.3cm

Now, we state the assumptions on the coefficients and collect some preliminary results from \cite{Budhiraja-Dupuis-Ganguly}, which will be used in the sequel.

\vskip0.3cm {\bf Condition A:}\ \
The coefficient  $G: H\times \mathbb{X}\rightarrow H$ and the force $f$  satisfy the following hypotheses:
\begin{itemize}
  \item[(A.1)] for some $L_G\in L^2(\vartheta)$,
  \beq\label{eq G L}
     |G(x_1,y)-G(x_2,y)|\leq L_G(y)|x_1-x_2|,\ \ \ x_1,\ x_2\in H,\ \ y\in\mathbb{X};
  \nneq
  \item[(A.2)] for some $M_G\in L^2(\vartheta)$,
  \beq\label{eq G M}
     |G(x,y)|\leq M_G(y)(1+|x|),\ \ \ x\in H,\ \ y\in\mathbb{X};
  \nneq

\item[(A.3)] $f\in L^2([0,T];V')$, i.e.,
\beq\label{eq f}
\int_0^T\|f(s)\|_{V'}^2ds<\infty.
\nneq
\end{itemize}

The following result follows by standard arguments (see \cite{B-Liu-Zhu}, \cite{Temam}).
\begin{thm}
Fix $u_0\in H$, and assume {\bf Condition A}. Let $u^\e$ be the unique solution of equation \eqref{SNSE-01} in $L^2(\Omega;D([0,T];H))\cap L^2(\Omega\times[0,T];V)$,
and $u^0$  the unique solution of equation \eqref{NSE} in $C([0,T],H)\cap L^2([0,T],V)$. Then, the following estimates hold: there exists
$\e_0>0$ such that
\begin{eqnarray}\label{Es u e}
\sup_{\e\in(0,\e_0]}\Big[\mathbb{E}\Big(\sup_{t\in[0,T]}|u^{\e}(t)|^2\Big)+\mathbb{E}\Big(\int_0^T\|u^{\e}(t)\|^2dt\Big)\Big]\leq C_{f,T,u_0};
\end{eqnarray}
and
\begin{eqnarray}\label{Es u 0}
\sup_{t\in[0,T]}|u^0(t)|^2+\int_0^T\|u^0(t)\|^2dt\leq C_{f,T,u_0}.
\end{eqnarray}
\end{thm}
\vskip 0.5cm

We now state a LDP for $\{Y^{\e}\}$ (namely, the MDP for $u^\e, \e>0$), where
\begin{eqnarray}\label{Eq Y e}
Y^{\e}=(u^\e-u^0)/a(\e),
\end{eqnarray}
and $a(\e)$ is as in (\ref{h}). To this end, we need to impose one more condition which will be stated below.

We define a class of functions by
\begin{eqnarray*}\label{Fun h}
\mathcal{H}=\Big\{h:\mathbb{X}\to\mathbb{R}:\ \exists\delta>0, s.t.\ \forall\Gamma\ with\ \vartheta(\Gamma)<\infty,\ \int_\Gamma\exp(\delta h^2(y))\vartheta(dy)<\infty\Big\}.
\end{eqnarray*}

\vskip0.3cm

{\bf Condition B:} The functions $L_G$ and $M_G$ are in the class $\mathcal{H}$.
\vskip0.3cm

The following theorem is our main result.
\begin{thm}\label{thm main}
Suppose that Conditions A and B hold. Then $\{Y^\e\}$ satisfies a large deviation principle in $D([0,T],H)\cap L^2([0,T],V)$
with speed $\e/{a^2(\e)}$ and the rate function given by
$$
I(\eta)=\inf_{\psi}\Big\{\frac{1}{2}\|\psi\|^2_2\Big\},
$$
where the infimum is taken over all $\psi\in L^2(\vartheta_T)$ such that $(\eta,\psi)$ satisfies the following equation
\begin{eqnarray}\label{Eq S I}
\frac{d}{dt}\eta(t)&=&-A\eta(t)-B(\eta(t),u^0(t))-B(u^0(t),\eta(t))\nonumber\\
&&+\int_{\mathbb{X}}\psi(y,t)G(u^0(t),y)\vartheta(dy),
\end{eqnarray}
with initial value $\eta(0)=0$.
\end{thm}
\begin{proof}
\noindent{\bf Proof of Theorem \ref{thm main}}

According to Theorem \ref{thm MDP}, it suffices to prove that \textbf{Condition MDP} is fulfilled. The verification of Condition \textbf{MDP-1} will be given by Proposition \ref{Prop 01}. Condition \textbf{MDP-2} will be established in Proposition \ref{prop 02}.
\end{proof}

\vspace{2mm}

Let $\{T_t,t\geq0\}$ denote the semigroup generated by $-A$. It is easy to see that $T_t,t\geq0$ are
    compact operators. For $f\in L^1([0,T],H)$, define the mapping
    \begin{eqnarray*}
        Rf(t)=\int_0^tT_{t-s}f(s)ds,\ t\geq0,
    \end{eqnarray*}
    which is the mild solution of the equation:
    \begin{eqnarray*}
       Z(t)=-\int_0^tAZ(s)ds+\int_0^tf(s)ds.
    \end{eqnarray*}

We recall the following lemma proved in \cite{Rockner-Zhang} (see Proposition 5.4 there).
\begin{lem}\label{lem-thm2-01}
    If $\mathcal{D}\subset L^1([0,T],H)$ is uniformly integrable, then the image family $\mathcal{Y}=R(\mathcal{D})$ is
    relatively compact in $C([0,T],H)$.
\end{lem}
\vskip 0.5cm

Denote $\mathcal{G}_0:\ L^2(\vartheta_T)\to C([0,T],H)\cap L^2([0,T],V)$ by
\begin{eqnarray}\label{Eq G 0}
\mathcal{G}_0(\psi)=\eta\ \text{if\ for\ }\psi\in L^2(\vartheta_T),\ \text{where} \ \eta\ \text{solves\ }(\ref{Eq S I}).
\end{eqnarray}

\begin{prop}\label{Prop 01}
Suppose that Conditions A and B hold. Fix $\Upsilon\in(0,\infty)$ and $g^\e,g\in B_2(\Upsilon)$ such that
$g^\e\to g$. Then $\mathcal{G}_0(g^\e)\to\mathcal{G}_0(g)\ \text{in}\ C([0,T],H)\cap L^2([0,T],V)$.
\end{prop}
\begin{proof}
Set \[f^\e(t)=\int_\mathbb{X}g^\e(y,t)G(u^0(t),y)\vartheta(dy), \qquad t\in[0,T].\]
By (\ref{eq G M}), we have
\begin{eqnarray*}
    \int_0^T\int_\mathbb{X}|G(u^0(t),y)|^2\vartheta(dy)dt
&\leq&
    \int_\mathbb{X}M^2_G(y)\vartheta(dy)\int_0^T(1+|u^0(t)|)^2dt\\
&\leq&
2T\sup_{t\in[0,T]}(1+|u^0(t)|^2)\int_\mathbb{X}M^2_G(y)\vartheta(dy)\\
&<&
\infty,
\end{eqnarray*}
and hence, for every $v\in H$, $\< G(u^0(t),y), v\>\in
L^2(\vartheta_T)$. Combining $g^\e\to g$ in the weak topology on
$L^2(\vartheta_T)$, we get
\begin{eqnarray}\label{Eq 01}
\lim_{\e\to 0}\< \int_0^tf^\e(s)ds, v\>=\<
\int_0^t\int_\mathbb{X}g(y,s)G(u^0(s),y)\vartheta(dy)ds, v\>,\
\ \forall v\in H,\ \forall t\in[0,T].
\end{eqnarray}

Denote $\mathcal{D}=\{f^\e,\ \e>0\}$. Since, for every measurable subset $O\subset[0,T]$
\begin{eqnarray}\label{Eq I-1}
    \int_O|f^\e(t)| dt
&\leq&
    \int_O\int_\mathbb{X}|g^\e(y,t)||G(u^0(t),y)| \vartheta(dy)dt\nonumber\\
&\leq&
    \Big(\int_0^T\int_\mathbb{X}|g^\e(y,t)|^2\vartheta(dy)dt\Big)^{1/2} \Big(\int_O\int_\mathbb{X}|G(u^0(t),y)|^2 \vartheta(dy)dt\Big)^{1/2}\nonumber\\
&\leq&
    \Upsilon\Big(\int_\mathbb{X}M(y)^2\vartheta(dy)\int_O(1+|u^0(t)| )^2dt\Big)^{1/2}\nonumber\\
&\leq&
    \Upsilon\sup_{t\in[0,T]}(1+|u^0(t)|)\sqrt{\lambda_T(O)},
\end{eqnarray}
we see that the family $\mathcal{D}\subset L^1([0,T],H)$  is uniformly integrable in $L^1([0,T],H)$. Therefore,
by Lemma \ref{lem-thm2-01}, $\{Z^\e,\ \e>0\}$ is relatively compact in $C([0,T],H)$, here $Z^\e$ satisfies
\begin{eqnarray*}
dZ^\e(t)=-AZ^\e(t)dt+f^\e(t)dt,\ \ t\in[0,T],
\end{eqnarray*}
with initial value $Z^\e(0)=0$.

Let $Z$ be any limit point of $\{Z^\e,\ \e>0\}$ in $C([0,T],H)$. Combining with (\ref{Eq 01}), we have
\begin{eqnarray*}
\< Z(t),v\>=-\int_0^t\< Z(s),Av\> ds+\< \int_0^t\int_\mathbb{X}g(y,s)G(u^0(s),y)\vartheta(dy)ds, v\> ,\ \ \forall v\in D(A).
\end{eqnarray*}
This implies that $Z$ is the unique solution of the following equation
\begin{eqnarray*}
\left\{
 \begin{array}{lll}
 & \hbox{ $dZ(t)=-AZ(t)dt+\int_\mathbb{X}g(y,t)G(u^0(t),y)\vartheta(dy)dt,\ \ t\in[0,T]$;} \\
 & \hbox{$\ \ Z(0)=0$.}
 \end{array}
\right.
\end{eqnarray*}

Denote $\overline{Z^\e}(t)=Z^\e(t)-Z(t)$. Notice that (\ref{Eq I-1}) also holds for \[f(t)=\int_\mathbb{X}g(y,t)G(u^0(t),y)\vartheta(dy)\] and $\sup_{s\in[0,T]}|\overline{Z^\e}(s)|\to 0,\ as\ \e\to 0$,
we obtain
\begin{eqnarray}\label{eq Z}
& &|\overline{Z^\e}(t)|^2+2\nu\int_0^t\|\overline{Z^\e}(s)\|^2ds\nonumber\\
&=&
2\int_0^t\< \overline{Z^\e}(s),\int_\mathbb{X}(g^\e(y,s)-g(y,s))G(u^0(s),y)\vartheta(dy)\> ds\nonumber\\
&\leq&
2\sup_{s\in[0,T]}|\overline{Z^\e}(s)|\nonumber\\
   & & \times \Big[
    \int_0^T\int_\mathbb{X}|g^\e(y,s)||G(u^0(s),y)| \vartheta(dy)ds
      +
    \int_0^T\int_\mathbb{X}|g(y,s)||G(u^0(s),y)| \vartheta(dy)ds
   \Big]\nonumber\\
&\leq&4\Upsilon\sup_{t\in[0,T]}(1+|u^0(t)|)\sqrt{T}\sup_{s\in[0,T]}|\overline{Z^\e}(s)|\to 0,\ as\ \e\to 0.
\end{eqnarray}

Set \[L^\e(t)=\mathcal{G}_0(g^\e)(t)-Z^\e(t) \mbox{ and }L(t)=\mathcal{G}_0(g)(t)-Z(t),\] and denote $\overline{L^\e}(t)=L^\e(t)-L(t)$. Then
\begin{eqnarray*}
\left\{
 \begin{array}{lll}
 & \hbox{ $d\overline{L^\e}(t)=-A\overline{L^\e}(t)dt-B(\overline{L^\e}(t)+\overline{Z^\e}(t),u^0(t))dt-B(u^0(t),\overline{L^\e}(t)+\overline{Z^\e}(t))dt$;} \\
 & \hbox{$\overline{L^\e}(0)=0$.}
 \end{array}
\right.
\end{eqnarray*}
We have
\begin{eqnarray}\label{eq L}
&&|\overline{L^\e}(t)|^2+2\nu\int_0^t\|\overline{L^\e}(s)\|^2ds\nonumber\\
&=&
 -2\int_0^t\left ( B(\overline{L^\e}(s)+\overline{Z^\e}(s),u^0(s)), \overline{L^\e}(s)\right ) ds\nonumber\\
&&-
 2\int_0^t\left ( B(u^0(s),\overline{L^\e}(s)+\overline{Z^\e}(s)), \overline{L^\e}(s)\right ) ds\nonumber\\
 &=&
 2\int_0^t\left ( B(\overline{L^\e}(s),\overline{L^\e}(s)), u^0(s)\right ) ds\nonumber\\
 &&-
 2\int_0^t\left ( B(\overline{Z^\e}(s),u^0(s)), \overline{L^\e}(s)\right ) ds\nonumber\\
&&-
 2\int_0^t\left ( B(u^0(s),\overline{Z^\e}(s)), \overline{L^\e}(s)\right ) ds\nonumber\\
 &=&
 I_1(t)+I_2(t)+I_3(t).
\end{eqnarray}

By (\ref{b6}) and (\ref{L4}),
\begin{eqnarray}\label{eq I01}
|I_1(t)|
&\leq&
2\int_0^t\Big|\left ( B(\overline{L^\e}(s),\overline{L^\e}(s)), u^0(s)\right )\Big|ds\nonumber\\
&\leq&
 \nu\int_0^t\|\overline{L^\e}(s)\|^2ds
+
 \frac{128}{\nu^3}\sup_{s\in[0,T]}|u^0(s)|^2\int_0^t\|u^0(s)\|^2|\overline{L^\e}(s)|^2ds.
\end{eqnarray}

By (\ref{b4}) and (\ref{Es u 0}),
\begin{eqnarray}\label{eq I02}
|I_2(t)|
&\leq&
2\int_0^t\Big|\left ( B(\overline{Z^\e}(s),u^0(s)), \overline{L^\e}(s)\right )\Big|ds\nonumber\\
&\leq&
4\int_0^t|\overline{Z^\e}(s)|^{1/2}\|\overline{Z^\e}(s)\|^{1/2}|u^0(s)|^{1/2}\|u^0(s)\|^{1/2}\|\overline{L^\e}(s)\|ds\nonumber\\
&\leq&
4\sup_{s\in[0,T]}|\overline{Z^\e}(s)|^{1/2}\sup_{s\in[0,T]}|u^0(s)|^{1/2}\int_0^t\|\overline{Z^\e}(s)\|^{1/2}\|u^0(s)\|^{1/2}\|\overline{L^\e}(s)\|ds\nonumber\\
&\leq&
C\sup_{s\in[0,T]}|\overline{Z^\e}(s)|^{1/2}
\Big[
   \int_0^t\|\overline{L^\e}(s)\|^2ds
     +
   \int_0^t\|\overline{Z^\e}(s)\|\|u^0(s)\|ds\nonumber
\Big]\\
&\leq&
C\sup_{s\in[0,T]}|\overline{Z^\e}(s)|^{1/2}
\Big[
   \int_0^t\|\overline{L^\e}(s)\|^2ds
     +
   C\Big(\int_0^t\|\overline{Z^\e}(s)\|^2ds\Big)^{1/2}
\Big].
\end{eqnarray}
Similar to (\ref{eq I02}), we have
\begin{eqnarray}\label{eq I03}
|I_3(t)|
\leq
C\sup_{s\in[0,T]}|\overline{Z^\e}(s)|^{1/2}
\Big[
   \int_0^t\|\overline{L^\e}(s)\|^2ds
     +
   C\Big(\int_0^t\|\overline{Z^\e}(s)\|^2ds\Big)^{1/2}
\Big].
\end{eqnarray}
Combining (\ref{eq L})--(\ref{eq I03}), we get
\begin{eqnarray}\label{eq LL}
&&|\overline{L^\e}(t)|^2+(\nu-C\sup_{s\in[0,T]}|\overline{Z^\e}(s)|^{1/2})\int_0^t\|\overline{L^\e}(s)\|^2ds\\
&\leq&
\frac{128}{\nu^3}\sup_{s\in[0,T]}|u^0(s)|^2\int_0^t\|u^0(s)\|^2|\overline{L^\e}(s)|^2ds
+
C\sup_{s\in[0,T]}|\overline{Z^\e}(s)|^{1/2}\Big(\int_0^T\|\overline{Z^\e}(s)\|^2ds\Big)^{1/2}.\nonumber
\end{eqnarray}
By (\ref{Es u 0}), (\ref{eq Z}) and using Gronwall's lemma,
\begin{eqnarray}\label{eq LLL}
\lim_{\e\to 0}\left \{\sup_{t\in[0,T]}|\overline{L^\e}(t)|^2+\int_0^T\|\overline{L^\e}(t)\|^2dt\right \}=0.
\end{eqnarray}
Recall \[L^\e(t)=\mathcal{G}_0(g^\e)(t)-Z^\e(t)\mbox{ and }L(t)=\mathcal{G}_0(g)(t)-Z(t).\]
(\ref{eq Z}) and (\ref{eq LLL}) yield that
\begin{eqnarray*}
\lim_{\e\to 0}\left \{\sup_{t\in[0,T]}|\mathcal{G}_0(g^\e)(t)-\mathcal{G}_0(g)(t)|^2
+\int_0^T\|\mathcal{G}_0(g^\e)(t)-\mathcal{G}_0(g)(t)\|^2dt\right\}=0.
\end{eqnarray*}

\end{proof}

Finally, we proceed to verifying Condition MDP-2.
Recall the definition of $\mathcal{U}^M_{+,\e}$ in $(\ref{eq u})$. We note that for every $\varphi^\e\in\mathcal{U}^M_{+,\e}$, there exists
unique process $X^\e\in D([0,T],H)\cap L^2([0,T],V)$ that solves the following equation
\[\left\{
 \begin{array}{ccl}
  dX^\e(t)&=&-AX^\e(t)dt-B(X^\e(t))dt+f(t)dt+\int_{\mathbb{X}}\e G(X^\e(t-),y)\widetilde{N}^{\e^{-1}\varphi^\e}(dydt)\\
 & & +\int_{\mathbb{X}}G(X^\e(t),y)(\varphi^\e(y,t)-1)\vartheta(dy)dt; \\
 X^\e(0)&=&u_0.
 \end{array}\right.\]

The following Lemmas \ref{lem 4.2}-\ref{lem 4.7} were proved in \cite{Budhiraja-Dupuis-Ganguly}. We refer the reader to \cite{Budhiraja-Dupuis-Ganguly} for details.
\begin{lem}\label{lem 4.2}
Let $h\in L^2(\vartheta)\cap\mathcal{H}$ and fix $M>0$.
Then there exists $\varsigma_h>0$ such that for any measurable subset $I$ of [0,T] and for all $\e>0$,
\begin{eqnarray}\label{eq lem 4.2}
\sup_{\varphi\in\mathcal{S}_{+,\e}^M}\int_{\mathbb{X}\times I}h^2(y)\varphi(y,s)\vartheta(dy)ds
\leq
\varsigma_h(a^2(\e)+\lambda_T(I)).
\end{eqnarray}
\end{lem}

\begin{lem}\label{lem 4.3}
Let $h\in L^2(\vartheta)\cap\mathcal{H}$ and $I$ be a measurable subset of [0,T]. Fix $M>0$. Then
there exists $\Gamma_h, \rho_h:(0,\infty)\to(0,\infty)$ such that $\Gamma_h(u)\downarrow 0$ as $u\uparrow \infty$, and for all
$\e,\beta\in(0,\infty)$,
$$
\sup_{\psi\in\mathcal{S}^M_\e}\int_{\mathbb{X}\times I}|h(y)\psi(y,s)|1_{\{|\psi|\geq \beta/a(\e)\}}\vartheta(dy)ds
\leq
\Gamma_h(\beta)(1+\sqrt{\lambda_T(I)}),
$$
and
$$
\sup_{\psi\in\mathcal{S}^M_\e}\int_{\mathbb{X}\times I}|h(y)\psi(y,s)|\vartheta(dy)ds
\leq
\rho_h(\beta)\sqrt{\lambda_T(I)}+\Gamma_h(\beta)a(\e).
$$

\end{lem}

\begin{lem}\label{lem 4.7}
Let $h\in L^2(\vartheta)\cap\mathcal{H}$ be positive. Then for any
$\beta>0$,
\begin{eqnarray}\label{eq lem 4.7}
\lim_{\e\to 0}\sup_{\psi\in\mathcal{S}^M_\e}\int_{\mathbb{X}\times [0,T]}|h(y)\psi(y,s)|1_{\{|\psi|>\beta/a(\e)\}}\vartheta(dy)ds
=
0.
\end{eqnarray}
\end{lem}

\begin{prop}
There exists an $\e_0>0$ such that
\begin{eqnarray}\label{es X e}
\sup_{\e\in(0,\e_0]}\Big(\mathbb{E}\sup_{t\in[0,T]}|X^\e(t)|^2+\mathbb{E}\int_0^T\|X^\e(t)\|^2dt\Big)
\leq
C_{\e_0}<\infty.
\end{eqnarray}
\end{prop}
\begin{proof}
By It\^{o}'s formula,
\begin{eqnarray}\label{es X 00}
&&d|X^\e(t)|^2+2\nu\|X^\e(t)\|^2dt\nonumber\\
&=&
 2 (f(t),X^\e(t))dt
+
 2\< \int_{\mathbb{X}}\e G(X^\e(t-),y)\widetilde{N}^{\e^{-1}\varphi^\e}(dydt),X^\e(t-)\>\\
&&+
 2\<\int_{\mathbb{X}}G(X^\e(t),y)(\varphi^\e(y,t)-1)\vartheta(dy)dt,X^\e(t)\>
+
 \int_{\mathbb{X}}\e^2 |G(X^\e(t-),y)|^2N^{\e^{-1}\varphi^\e}(dydt).\nonumber
\end{eqnarray}

We have
\begin{eqnarray}\label{es X 01}
\int_0^t|2(f(s),X^\e(s))|ds
\leq
\nu\int_0^t\|X^\e(s)\|^2ds+\frac{1}{\nu}\int_0^t\|f(s)\|^2_{V'}ds.
\end{eqnarray}
Set $\psi^\e(y,s)=(\varphi^\e(y,s)-1)/a(\e)\in \mathcal{U}^M_\e$. Then
\begin{eqnarray}\label{es X 02}
&&\Big|\int_0^t2\<\int_{\mathbb{X}}G(X^\e(s),y)(\varphi^\e(y,s)-1)\vartheta(dy)ds,X^\e(s)
\>\Big|\nonumber\\
&\leq&
2a(\e)\int_0^t|X^\e(s)|\int_{\mathbb{X}}|G(X^\e(s),y)||\psi^\e(y,s)|\vartheta(dy)ds\nonumber\\
&\leq&
2a(\e)\int_0^t|X^\e(s)|(1+|X^\e(s)|)\int_{\mathbb{X}}M_G(y)|\psi^\e(y,s)|\vartheta(dy)ds\nonumber\\
&\leq&
4a(\e)\int_0^t(1+|X^\e(s)|^2)\int_{\mathbb{X}}M_G(y)|\psi^\e(y,s)|\vartheta(dy)ds.
\end{eqnarray}
Combining (\ref{es X 00})--(\ref{es X 02}), we have
\begin{eqnarray}\label{es X 04}
&&|X^\e(t)|^2+\nu\int_0^t\|X^\e(s)\|^2ds\nonumber\\
&\leq&
|u_0|^2
+
\frac{1}{\nu}\int_0^T\|f(s)\|^2_{V'}ds
+
 \sup_{l\in[0,T]}\Big|2\int_0^l\< \int_{\mathbb{X}}\e G(X^\e(s-),y)\widetilde{N}^{\e^{-1}\varphi^\e}(dyds),X^\e(s-)\>\Big|\nonumber\\
&&+
 \int_0^T\int_{\mathbb{X}}\e^2 |G(X^\e(s-),y)|^2N^{\e^{-1}\varphi^\e}(dyds)
 +
 4a(\e)\int_0^T\int_{\mathbb{X}}M_G(y)|\psi^\e(y,s)|\vartheta(dy)ds\nonumber\\
&& +
 4a(\e)\int_0^t|X^\e(s)|^2\int_{\mathbb{X}}M_G(y)|\psi^\e(y,s)|\vartheta(dy)ds\nonumber\\
 &=&
 I_1+I_2+I_3+I_4+I_5+I_6(t).
\end{eqnarray}
Applying Gronwall' lemma and using Lemma \ref{lem 4.3}, we get
\begin{eqnarray}\label{es X 05}
&&|X^\e(t)|^2+\int_0^t\|X^\e(s)\|^2ds\nonumber\\
&\leq&
\Big(I_1+I_2+I_3+I_4+I_5\Big)\exp\Big[4a(\e)\Big(\rho_{M_G}(\beta)\sqrt{T}+\Gamma_{M_G}(\beta)a(\e)\Big)\Big].
\end{eqnarray}

By (\ref{eq f}) and Lemma \ref{lem 4.3},
\begin{eqnarray}\label{es X 06}
I_1+I_2+I_5
\leq
C+4a(\e)\Big(\rho_{M_G}(\beta)\sqrt{T}+\Gamma_{M_G}(\beta)a(\e)\Big).
\end{eqnarray}
By Burkholder-Davis-Gundy   inequality and Lemma \ref{lem 4.2},
\begin{eqnarray}\label{es X 07}
\mathbb{E}I_3
&\leq&
\mathbb{E}\Big(\int_0^T\int_{\mathbb{X}}4\e^2|X^\e(s-)|^2|G(X^\e(s-),y)|^2N^{\e^{-1}\varphi^\e}(dy, ds)\Big)^{1/2}\nonumber\\
&\leq&
\mathbb{E}\Big[\sup_{s\in[0,T]}|X^\e(s)|\Big(\int_0^T\int_{\mathbb{X}}4\e^2|G(X^\e(s-),y)|^2N^{\e^{-1}\varphi^\e}(dy, ds)\Big)^{1/2}\Big]\nonumber\\
%&\leq&
%\frac{1}{4}\mathbb{E}\sup_{s\in[0,T]}|X^\e(s)|^2
%+
%16\e^2\mathbb{E}\Big(\int_0^T\int_{\mathbb{X}}|G(X^\e(s-),y)|^2N^{\e^{-1}\varphi^\e}(dy, ds)\Big)\nonumber\\
%&\leq&
%\frac{1}{4}\mathbb{E}\sup_{s\in[0,T]}|X^\e(s)|^2
%+
%16\e\mathbb{E}\Big(\int_0^T\int_{\mathbb{X}}|G(X^\e(s),y)|^2\varphi^\e(y,s) \vartheta(dy)ds\Big)\nonumber\\
&\leq&
\frac{1}{4}\mathbb{E}\sup_{s\in[0,T]}|X^\e(s)|^2
+
16\e\mathbb{E}\Big(\int_0^T\int_{\mathbb{X}}|G(X^\e(s),y)|^2\varphi^\e(y,s) \vartheta(dy)ds\Big)\nonumber\\
&\leq&
\frac{1}{4}\mathbb{E}\sup_{s\in[0,T]}|X^\e(s)|^2
+
32\e\mathbb{E}\Big((\sup_{s\in[0,T]}|X^\e(s)|^2+1)\int_0^T\int_{\mathbb{X}}M^2_G(y)\varphi^\e(y,s) \vartheta(dy)ds\Big)\nonumber\\
&\leq&
\frac{1}{4}\mathbb{E}\sup_{s\in[0,T]}|X^\e(s)|^2
+
32\e\varsigma_{M_G}(a^2(\e)+T)\mathbb{E}(\sup_{s\in[0,T]}|X^\e(s)|^2+1).
\end{eqnarray}
Similar to (\ref{es X 07}), we get
\begin{eqnarray}\label{es X 08}
\mathbb{E}I_4
&=&
\e\mathbb{E}\int_0^T\int_{\mathbb{X}} |G(X^\e(s),y)|^2\varphi^\e(y,s)\vartheta(dy)ds\nonumber\\
&\leq&
2\e\varsigma_{M_G}(a^2(\e)+T)\mathbb{E}(\sup_{s\in[0,T]}|X^\e(s)|^2+1).
\end{eqnarray}

Choosing $\e_0>0$ such that $34\e_0\varsigma_{M_G}(a^2(\e_0)+T)\leq 1/8$, and combining (\ref{es X 05})--(\ref{es X 08}),
we obtain (\ref{es X e}). The proof is complete.

\end{proof}

Recall $(\ref{NSE})$. We have
\begin{thm}
\begin{eqnarray}\label{lim Xe U0}
\lim_{\e\to 0}\Big(\mathbb{E}\sup_{t\in[0,T]}|X^\e(t)-u^0(t)|^2+\mathbb{E}\int_0^T\|X^\e(t)-u^0(t)\|^2dt\Big)
=
0.
\end{eqnarray}
\end{thm}
\begin{proof}
Set $Z^\e(t)=X^\e(t)-u^0(t)$. Then
\begin{eqnarray}\label{eq Z e}
dZ^\e(t)
&=&
-AZ^\e(t)dt-B(X^\e(t),Z^\e(t))dt-B(Z^\e(t),u^0(t))dt\\
&+&
\e\int_{\mathbb{X}}G(X^\e(t-),y)\tilde{N}^{\e^{-1}\varphi^{\e}}(dydt)
+
\int_{\mathbb{X}}G(X^\e(t),y)(\varphi^\e(y,t)-1)\vartheta(dy)dt\nonumber
\end{eqnarray}
with initial value $Z^\e(0)=0$.

Apply Ito's Formula,
\begin{eqnarray}\label{eq z 01}
& &d|Z^\e(t)|^2+2\nu\|Z^\e(t)\|^2dt\\
&=&
2\< B(Z^\e(t),Z^\e(t)),u^0(t)\> dt
+
2\e\int_{\mathbb{X}}\< G(X^\e(t-),y), Z^\e(t-)\>\tilde{N}^{\e^{-1}\varphi^{\e}}(dydt)\nonumber\\
& &+
2\int_{\mathbb{X}}\< G(X^\e(t),y)(\varphi^\e(y,t)-1), Z^\e(t)\>\vartheta(dy)dt\nonumber
+
\e^2\int_{\mathbb{X}}|G(X^\e(t-),y)|^2N^{\e^{-1}\varphi^{\e}}(dydt).
\end{eqnarray}

By (\ref{b6}) and (\ref{L4}),
\begin{eqnarray}\label{eq z 02}
& &\int_0^t2\Big|\< B(Z^\e(s),Z^\e(s)),u^0(s)\>\Big| ds\nonumber\\
&\leq&
\nu\int_0^t\|Z^\e(s)\|^2ds+\frac{64}{\nu^3}\int_0^t\|u^0(s)\|^4_{L^4}|Z^\e(s)|^2ds\nonumber\\
&\leq&
\nu\int_0^t\|Z^\e(s)\|^2ds+\frac{64}{\nu^3}\sup_{l\in[0,T]}|u^0(l)|^2\int_0^t\|u^0(s)\|^2|Z^\e(s)|^2ds.
\end{eqnarray}
Set $\psi^\e(y,t)=(\varphi^\e(y,t)-1)/a(\e)$. By (\ref{eq G L}) and (\ref{eq G M}),
\begin{eqnarray}\label{eq z 03}
& &2\int_0^t\Big|\int_{\mathbb{X}}\< G(X^\e(s),y)(\varphi^\e(y,s)-1), Z^\e(s)\>\vartheta(dy)\Big|ds\nonumber\\
&\leq&
2\int_0^t|Z^\e(s)|\int_{\mathbb{X}}|G(X^\e(s),y)-G(u^0(s),y)||\varphi^\e(y,s)-1|\vartheta(dy)ds\nonumber\\
&&+
2\int_0^t|Z^\e(s)|\int_{\mathbb{X}}|G(u^0(s),y)||\varphi^\e(y,s)-1|\vartheta(dy)ds\nonumber\\
&\leq&
2a(\e)\int_0^t|Z^\e(s)|^2\int_{\mathbb{X}}L_G(y)|\psi^\e(y,s)|\vartheta(dy)ds\nonumber\\
&&+
a(\e)\int_0^t(1+|Z^\e(s)|^2)(1+|u^0(s)|)\int_{\mathbb{X}}M_G(y)|\psi^\e(y,s)|\vartheta(dy)ds\nonumber\\
&\leq&
a(\e)\int_0^t|Z^\e(s)|^2\int_{\mathbb{X}}\Big(2L_G(y)+(1+\sup_{l\in[0,T]}|u^0(l)|)M_G(y)\Big)|\psi^\e(y,s)|\vartheta(dy)ds\nonumber\\
&&+
a(\e)(1+\sup_{l\in[0,T]}|u^0(l)|)\int_0^t\int_{\mathbb{X}}M_G(y)|\psi^\e(y,s)|\vartheta(dy)ds.
\end{eqnarray}
Combining (\ref{eq z 01})-(\ref{eq z 03}), we get
\begin{eqnarray*}
|Z^\e(t)|^2+\nu\int_0^t\|Z^\e(s)\|^2ds
\leq
M_1(T)+M_2(T)+M_3(T)+\int_0^t J(s)|Z^\e(s)|^2ds,
\end{eqnarray*}
here
\begin{eqnarray*}
M_1(T)
=
2\e\sup_{s\in[0,T]}\Big|\int_0^s\int_{\mathbb{X}}\< G(X^\e(l-),y), Z^\e(l-)\>\tilde{N}^{\e^{-1}\varphi^{\e}}(dydl)\Big|,
\end{eqnarray*}
\begin{eqnarray*}
M_2(T)
=
\e^2\int_0^T\int_{\mathbb{X}}|G(X^\e(t-),y)|^2N^{\e^{-1}\varphi^{\e}}(dydt),
\end{eqnarray*}
\begin{eqnarray*}
M_3(T)
=
a(\e)(1+\sup_{l\in[0,T]}|u^0(l)|)\int_0^T\int_{\mathbb{X}}M_G(y)|\psi^\e(y,s)|\vartheta(dy)ds,
\end{eqnarray*}
and
\begin{eqnarray*}
J(s)&=&
\frac{64}{\nu^3}\sup_{l\in[0,T]}|u^0(l)|^2\|u^0(s)\|^2
+
2a(\e)\int_{\mathbb{X}}L_G(y)|\psi^\e(y,s)|\vartheta(dy)\\
&&+
a(\e)(1+\sup_{l\in[0,T]}|u^0(l)|)\int_{\mathbb{X}}M_G(y)|\psi^\e(y,s)|\vartheta(dy).
\end{eqnarray*}
By Gronwall' lemma, Lemma \ref{lem 4.3} and (\ref{Es u 0}),
\begin{eqnarray}\label{es M}
& &|Z^\e(t)|^2+\nu\int_0^t\|Z^\e(s)\|^2ds\nonumber\\
&\leq&
\Big(M_1(T)+M_2(T)+M_3(T)\Big)\exp\Big(\int_0^T J(s)ds\Big)\nonumber\\
&\leq&
C\Big(M_1(T)+M_2(T)+M_3(T)\Big).
\end{eqnarray}

By Lemma \ref{lem 4.2} and (\ref{es X e})
\begin{eqnarray}\label{es M1}
&&\mathbb{E}M_1(T)\nonumber\\
&\leq&
\mathbb{E}\Big(\int_0^T\int_{\mathbb{X}}4\e^2|G(X^\e(l-),y)|^2|Z^\e(l-)|^2 N^{\e^{-1}\varphi^{\e}}(dydl)\Big)^{1/2}\nonumber\\
&\leq&
1/2\mathbb{E}\Big(\sup_{t\in[0,T]}|Z^\e(t)|^2\Big)+8\e\mathbb{E}\Big(\int_0^T\int_{\mathbb{X}}M^2_G(y)(1+|X^\e(l)|)^2 \varphi^\e(y,l)\vartheta(dy)dl\Big)\nonumber\\
&\leq&
1/2\mathbb{E}\Big(\sup_{t\in[0,T]}|Z^\e(t)|^2\Big)+8\e\mathbb{E}\Big(\sup_{t\in[0,T]}(1+|X^\e(t)|)^2\int_0^T\int_{\mathbb{X}}M^2_G(y) \varphi^\e(y,l)\vartheta(dy)dl\Big)\nonumber\\
&\leq&
1/2\mathbb{E}\Big(\sup_{t\in[0,T]}|Z^\e(t)|^2\Big)+16\e\varsigma_{M_G}(a^2(\e)+T)C.
\end{eqnarray}

Similarly, we have
\begin{eqnarray}\label{es M2}
\mathbb{E}M_2(T)
&=&
\e\mathbb{E}\int_0^T\int_{\mathbb{X}}|G(X^\e(t),y)|^2\varphi^{\e}(y,t)\vartheta(dy)dt\nonumber\\
&\leq&
2\e\mathbb{E}\Big(\sup_{t\in[0,T]}(1+|X^\e(t)|^2)\int_0^T\int_{\mathbb{X}}M^2_G(y) \varphi^\e(y,t)\vartheta(dy)dt\Big)\nonumber\\
&\leq&
\e\varsigma_{M_G}(a^2(\e)+T)C.
\end{eqnarray}

By (\ref{Es u 0}) and Lemma \ref{lem 4.3},
\begin{eqnarray}\label{es M3}
M_3(T)
\leq
Ca(\e)\Big(\rho_{M_G}(\beta)\sqrt{T}+\Gamma_{M_G}(\beta)a(\e)\Big).
\end{eqnarray}
Combining (\ref{es M})--(\ref{es M3}), we have
\begin{eqnarray}
\lim_{\e\to 0}\Big(\mathbb{E}\sup_{t\in[0,T]}|Z^\e(t)|^2+\mathbb{E}\int_0^T\|Z^\e(t)\|^2dt\Big)
=
0.
\end{eqnarray}
The proof is complete.
\end{proof}

\vskip 5mm
Define
\begin{eqnarray}\label{eq G e}
\mathcal{G}^\e(\e N^{\e^{-1}\varphi^\e}):=Y^\e=\frac{1}{a(\e)}(X^\e-u^0).
\end{eqnarray}
Then $Y^\e$ satisfies
\begin{eqnarray}\label{eq Ye0}
\left\{
 \begin{array}{llll}
 & \hbox{$dY^\e(t)=-AY^\e(t)dt-B(Y^\e(t),u^0(t))dt-B(X^\e(t),Y^\e(t))dt$}\\
 & \hbox{$\qquad\qquad +\frac{\e}{a(\e)}\int_{\mathbb{X}} G(X^\e(t-),y)\widetilde{N}^{\e^{-1}\varphi^\e}(dydt)$}\\
 & \hbox{$\qquad\qquad+\frac{1}{a(\e)}\int_{\mathbb{X}}G(X^\e(t),y)(\varphi^\e(y,t)-1)\vartheta(dy)dt$}, \\
 & \hbox{$Y^\e(0)=0$.}
 \end{array}
\right.
\end{eqnarray}

\begin{prop}\label{prop 02}
Given $M<\infty$. Let $\{\varphi^\e\}_{\e>0}$ be such that $\varphi^\e\in\mathcal{U}^M_{+,\e}$ for every $\e>0$. Let
$\psi^\e=(\varphi^\e-1)/a(\e)$ and $\beta\in(0,1]$. Then the family $\{Y^\e, \psi^\e1_{\{|\psi^\e|\leq \beta/a(\e)\}}\}_{\e>0}$ is tight in $D([0,T],H)\times B_2\Big(\sqrt{M\kappa_2(1)}\Big)$, and any limit point $(Y,\psi)$ solves the equation (\ref{Eq S I}).
\end{prop}

\begin{proof} The proof is divided into four steps.

{\bf Step 1.} Let $Z^\e$ be the solution of the following equation
\begin{eqnarray*}
dZ^\e(t)=-AZ^\e(t)dt+\frac{\e}{a(\e)}\int_{\mathbb{X}} G(X^\e(t-),y)\widetilde{N}^{\e^{-1}\varphi^\e}(dydt)
\end{eqnarray*}
with initial value $Z^\e(0)=0$.

Applying Ito's formula to $|Z^\e(t)|^2$,
\begin{eqnarray}\label{eq Z1}
& &d|Z^\e(t)|^2+2\nu\|Z^\e(t)\|^2dt\\
&=&
\frac{2\e}{a(\e)}\int_{\mathbb{X}}\< G(X^\e(t-),y), Z^\e(t-)\>\widetilde{N}^{\e^{-1}\varphi^\e}(dydt)
+
\frac{\e^2}{a^2(\e)}\int_{\mathbb{X}}|G(X^\e(t-),y)|^2N^{\e^{-1}\varphi^\e}(dydt).\nonumber
\end{eqnarray}
By Burkholder-Davis-Gundy inequality, Lemma \ref{lem 4.2} and (\ref{es X e}), we have
\begin{eqnarray}\label{eq Z2}
& &\mathbb{E}\Big(\sup_{t\in[0,T]}
\Big|
\int_0^t\int_{\mathbb{X}}\frac{2\e}{a(\e)}\< G(X^\e(s-),y), Z^\e(s-)\>\widetilde{N}^{\e^{-1}\varphi^\e}(dyds)
\Big|\Big)\nonumber\\
&\leq&
C\mathbb{E}\Big(\int_0^T\int_{\mathbb{X}}\frac{\e^2}{a^2(\e)}|G(X^\e(s-),y)|^2|Z^\e(s-)|^2N^{\e^{-1}\varphi^\e}(dyds)\Big)^{1/2}\nonumber\\
&\leq&
1/2\mathbb{E}\Big(\sup_{t\in[0,T]}|Z^\e(t)|^2\Big)
+
C
\mathbb{E}\Big(\int_0^T\int_{\mathbb{X}}\frac{\e^2}{a^2(\e)}|G(X^\e(s-),y)|^2N^{\e^{-1}\varphi^\e}(dyds)\Big)\nonumber\\
&\leq&
1/2\mathbb{E}\Big(\sup_{t\in[0,T]}|Z^\e(t)|^2\Big)
+
\frac{C\e}{a^2(\e)}\mathbb{E}\Big(\int_0^T\int_{\mathbb{X}}M^2_G(y)(1+|X^\e(s)|^2)\varphi^\e(y,s)\vartheta(dy)ds\Big)\nonumber\\
&\leq&
1/2\mathbb{E}\Big(\sup_{t\in[0,T]}|Z^\e(t)|^2\Big)
+
\frac{C\e}{a^2(\e)}\varsigma_{M_G}(a^2(\e)+T),
\end{eqnarray}
and similarly
\begin{eqnarray}\label{eq Z3}
& &\mathbb{E}\Big(\int_0^T\int_{\mathbb{X}}\frac{\e^2}{a^2(\e)}|G(X^\e(t-),y)|^2N^{\e^{-1}\varphi^\e}(dydt)\Big)\nonumber\\
&=&
\frac{\e}{a^2(\e)}\mathbb{E}\Big(\int_0^T\int_{\mathbb{X}}|G(X^\e(t),y)|^2\varphi^\e(y,t)\vartheta(dy)dt\Big)\nonumber\\
&\leq&
\frac{\e}{a^2(\e)}\varsigma_{M_G}(a^2(\e)+T).
\end{eqnarray}
Combining (\ref{eq Z1}) (\ref{eq Z2}) and (\ref{eq Z3}), we obtain
\begin{eqnarray}\label{Z}
\lim_{\e\to0}\mathbb{E}\Big(\sup_{t\in[0,T]}|Z^\e(t)|^2+\int_0^T\|Z^\e(t)\|^2dt\Big)=0.
\end{eqnarray}

{\bf Step 2.} Recall $\psi^\e=(\varphi^\e-1)/a(\e)$. Let $L^\e(t)$ be the unique solution of
\begin{eqnarray*}
\left\{
 \begin{array}{llll}
 & \hbox{ $dL^\e(t)=-AL^\e(t)dt+\int_{\mathbb{X}}G(X^\e(t),y)\psi^\e(y,t)1_{\{|\psi^\e|>\beta/a(\e)\}}\vartheta(dy)dt$},\\
 & \hbox{$L^\e(0)=0$.}
 \end{array}
\right.
\end{eqnarray*}

We have
\begin{eqnarray}\label{eq L1}
& &|L^\e(t)|^2+2\nu\int_0^t\|L^\e(s)\|^2ds\nonumber\\
&=&
2\int_0^t\int_{\mathbb{X}}\< G(X^\e(s),y)\psi^\e(y,s)1_{\{|\psi^\e|>\beta/a(\e)\}},L^\e(s)\>\vartheta(dy)ds\nonumber\\
&\leq&
2\int_0^T\int_{\mathbb{X}}|G(X^\e(s),y)||\psi^\e(y,s)|1_{\{|\psi^\e|>\beta/a(\e)\}}|L^\e(s)|\vartheta(dy)ds\nonumber\\
&\leq&
2\sup_{t\in[0,T]}|L^\e(t)|\sup_{t\in[0,T]}(1+|X^\e(t)|)\int_0^T\int_{\mathbb{X}}M_G(y)|\psi^\e(y,s)|1_{\{|\psi^\e|>\beta/a(\e)\}}\vartheta(dy)ds\nonumber\\
&\leq&
1/2\sup_{t\in[0,T]}|L^\e(t)|^2\nonumber\\
&&+C\sup_{t\in[0,T]}(1+|X^\e(t)|^2)\Big[\int_0^T\int_{\mathbb{X}}M_G(y)|\psi^\e(y,s)|1_{\{|\psi^\e|>\beta/a(\e)\}}\vartheta(dy)ds\Big]^2.\nonumber
\end{eqnarray}
By (\ref{es X e}) and Lemma \ref{lem 4.7},
\begin{eqnarray}\label{eq L1-a}
& &\mathbb{E}\Big(\sup_{t\in[0,T]}|L^\e(t)|^2+\nu\int_0^T\|L^\e(t)\|^2dt\Big)\nonumber\\
&\leq&
C\mathbb{E}\Big(\sup_{t\in[0,T]}(1+|X^\e(t)|^2)\Big)
\Big[\sup_{\psi\in\mathcal{S}^M_\e}\int_0^T\int_{\mathbb{X}}M_G(y)|\psi(y,s)|1_{\{|\psi|>\beta/a(\e)\}}\vartheta(dy)ds\Big]^2\nonumber\\
&\to & 0,\ \ as\ \e \to 0.
\end{eqnarray}

{\bf Step 3.} Denote by $U^\e$ the unique solution of the following equation
\begin{eqnarray*}\label{eq U0}
dU^\e(t)=-AU^\e(t)dt+\int_{\mathbb{X}}\Big(G(X^\e(t),y)-G(u^0(t),y)\Big)\psi^\e(y,t)1_{\{|\psi^\e|\leq\beta/a(\e)\}}\vartheta(dy)dt,
\end{eqnarray*}
with initial value $U^\e(0)=0$. Then
\begin{eqnarray*}\label{eq U1}
&  &|U^\e(t)|^2+2\nu\int_0^t\|U^\e(s)\|^2ds\\
&=&
2\int_0^t\int_{\mathbb{X}}
    \< \Big(G(X^\e(s),y)-G(u^0(s),y)\Big), U^\e(s)\>\psi^\e(y,s)1_{\{|\psi^\e|\leq\beta/a(\e)\}}
\vartheta(dy)ds\nonumber\\
&\leq&
2\int_0^T\int_{\mathbb{X}}
    \Big|G(X^\e(s),y)-G(u^0(s),y)\Big||U^\e(s)||\psi^\e(y,s)|
\vartheta(dy)ds\nonumber\\
&\leq&
2\sup_{s\in[0,T]}|U^\e(s)|\sup_{s\in[0,T]}|X^\e(s)-u^0(s)|\int_0^T\int_{\mathbb{X}}
    L_G(y)|\psi^\e(y,s)|
\vartheta(dy)ds\nonumber\\
&\leq&
1/2\sup_{s\in[0,T]}|U^\e(s)|^2+C\sup_{s\in[0,T]}\Big |X^\e(s)-u^0(s)\Big |^2\Big(\sup_{\psi\in\mathcal{S}^M_\e}\int_0^T\int_{\mathbb{X}}
    L_G(y)|\psi(y,s)|
\vartheta(dy)ds\Big)^2.\nonumber\\
\end{eqnarray*}
By Lemma \ref{lem 4.3} and (\ref{lim Xe U0}), we have
\begin{eqnarray}\label{eq U2}
\lim_{\e\to 0}\Big[\mathbb{E}\Big(\sup_{s\in[0,T]}|U^\e(s)|^2\Big)+\mathbb{E}\Big(\int_0^T\|U^\e(s)\|^2ds\Big)\Big]=0.
\end{eqnarray}

{\bf Step 4.}  Set $K^\e=Z^\e+L^\e+U^\e$ and denote $\Upsilon^\e=Y^\e-K^\e$. By (\ref{eq Ye0}), we have
\begin{eqnarray}\label{eq Up}
\left\{
 \begin{array}{llll}
 & \hbox{ $d\Upsilon^\e(t)=-A\Upsilon^\e(t)dt-a(\e)B\Big(\Upsilon^\e(t)+K^\e(t),\Upsilon^\e(t)+K^\e(t)\Big)dt$},\\
 & \hbox{ $\ \ \ \ \ \ \ \ \ \ \ \ \ -B\Big(u^0(t),\Upsilon^\e(t)+K^\e(t)\Big)dt-B\Big(\Upsilon^\e(t)+K^\e(t),u^0(t)\Big)dt$},\\
 & \hbox{ $\ \ \ \ \ \ \ \ \ \ \ \ \ +\int_{\mathbb{X}}G(u^0(t),y)\psi^\e(y,t)1_{\{|\psi^\e|\leq\beta/a(\e)\}}\vartheta(dy)dt$},\\

 & \hbox{$\Upsilon^\e(0)=0$.}
 \end{array}
\right.
\end{eqnarray}

Set
$$
\Pi=\Big(D([0,T],H)\cap L^2([0,T],V);\ C([0,T],H)\cap L^2([0,T],V);\ B_2\Big(\sqrt{M\kappa_2(1)}\Big)\Big)
.$$
By (\ref{Z}), (\ref{eq L1-a}) and (\ref{eq U2}), and notice that $(\psi^\e1_{\{|\psi^\e|\leq \beta/a(\e)\}})_{\e>0}$ is tight in $B_2\Big(\sqrt{M\kappa_2(1)}\Big)$ (see Lemma 3.2 in \cite{Budhiraja-Dupuis-Ganguly}),
$(Z^\e,\ L^\e+U^\e,\ \psi^\e1_{\{|\psi^\e|\leq \beta/a(\e)\}})_{\e>0}$ is tight in $\Pi$, and let $(0,0,\psi)$ be any limit point of the tight family,
and denote by  $Y=\mathcal{G}_0(\psi)$ the solution of equation (\ref{Eq S I}).

\vskip 0.3cm
It follows from the Skorokhod representation theorem that there exist a stochastic basis
$(\Omega^1,\mathcal{F}^1,\{\mathcal{F}_t^1\}_{t\in[0,T]},\mathbb{P}^1)$ and, on this basis,
$\Pi$-valued random variables $(\widetilde{Z}^\e,\ \widetilde{LU}^\e,\ \widetilde{\psi}^\e),$
$(0,0,\widetilde{\psi}),\ \epsilon\in(0,\epsilon_0)$,
such that $(\widetilde{Z}^\e,\ \widetilde{LU}^\e,\ \widetilde{\psi}^\e)$ (respectively $(0,0,\widetilde{\psi})$) has the same law as $(Z^\e,\ L^\e+U^\e,\ \psi^\e1_{\{|\psi^\e|\leq \beta/a(\e)\}})$ (respectively $(0,0,\psi)$), and
$(\widetilde{Z}^\e,\ \widetilde{LU}^\e,\ \widetilde{\psi}^\e)\rightarrow (0,0,\widetilde{\psi})$ in $\Pi,\ \mathbb{P}^1$-a.s..

Set $\widetilde{K}^\e=\widetilde{Z}^\e+\widetilde{LU}^\e$. Denote by $\widetilde{\Upsilon}^\e$ the unique solution of (\ref{eq Up}) with $(K^\e, \psi^\e)$
replaced by $(\widetilde{K}^\e, \widetilde{\psi}^\e)$. Then $(\widetilde{K}^\e, \widetilde{\Upsilon}^\e)$ has the law as $(K^\e, \Upsilon^\e)$. Hence, $\widetilde{Y}^\e=\widetilde{K}^\e+\widetilde{\Upsilon}^\e$ has the same law
 as $Y^\e=K^\e+\Upsilon^\e$ in $D([0,T],H)\cap L^2([0,T],V)$. Denote by $\widetilde{Y}$ the solution of equation (\ref{Eq S I}) with $\psi(y,t)$ replaced by $\widetilde{\psi}(y,t)$.
$\widetilde{Y}$ must have the same law as $Y$.
\vskip 0.4cm
Thus, the proof of the Proposition will be complete if we can show  that
\begin{eqnarray}\label{lim 01}
\sup_{t\in[0,T]}|\widetilde{Y}^\e(t)-\widetilde{Y}(t)|^2+\int_0^T\|\widetilde{Y}^\e(t)-\widetilde{Y}(t)\|^2dt\to 0,\ \mathbb{P}^1-a.s.,\ as\ \e\to 0.
\end{eqnarray}
This is the task of the remaining proof.
\vskip 0.4cm
Consider the following equation
\begin{eqnarray}\label{eq Ga}
\left\{
 \begin{array}{llll}
 & \hbox{ $d\widetilde{\Gamma}^\e(t)=-A\widetilde{\Gamma}^\e(t)dt+\int_{\mathbb{X}}G(u^0(t),y)\widetilde{\psi}^\e(y,t)\vartheta(dy)dt$},\\
 & \hbox{$\ \ \widetilde{\Gamma}^\e(0)=0$.}
 \end{array}
\right.
\end{eqnarray}
Using similar arguments as in the proof of (\ref{eq Z}), we have
\begin{eqnarray}\label{eq 00}
\lim_{\e\to 0}\Big(\sup_{t\in[0,T]}|\widetilde{\Gamma}^\e(t)-\widetilde{\Gamma}(t)|^2+\int_0^T\|\widetilde{\Gamma}^\e(t)-\widetilde{\Gamma}(t)\|^2dt\Big)
=
0,
\end{eqnarray}
here $\widetilde{\Gamma}$ satisfies (\ref{eq Ga}) with $\widetilde{\psi}^\e(y,t)$ replaced by $\widetilde{\psi}(y,t)$.

Set $\widetilde{M}=\widetilde{Y}-\widetilde{\Gamma}$ and $\widetilde{M}^\e=\widetilde{Y}^\e-\widetilde{K}^\e-\widetilde{\Gamma}^\e$. Then
\begin{eqnarray}\label{eq M0}
\left\{
 \begin{array}{llll}
 & \hbox{ $d\widetilde{M}(t)=-A\widetilde{M}(t)dt-B\Big(u^0(t),\widetilde{M}(t)+\widetilde{\Gamma}(t)\Big)dt-B\Big(\widetilde{M}(t)+\widetilde{\Gamma}(t),u^0(t)\Big)dt$},\\
 & \hbox{$\widetilde{M}(0)=0$.}
 \end{array}
\right.
\end{eqnarray}
and
\begin{equation}\label{eq M1}
\left\{
 \begin{array}{ccl}
 d\widetilde{M}^\e(t)&=&-A\widetilde{M}^\e(t)dt-a(\e)B\Big(\widetilde{M}^\e(t)
 +\widetilde{\Gamma}^\e(t)+\widetilde{K}^\e(t),\widetilde{M}^\e(t)+\widetilde{\Gamma}^\e(t)+\widetilde{K}^\e(t)\Big)dt,\\
  & & -B\Big(u^0(t),\widetilde{M}^\e(t)+\widetilde{\Gamma}^\e(t)+\widetilde{K}^\e(t)\Big)dt\\
  && -B\Big(\widetilde{M}^\e(t)
  +\widetilde{\Gamma}^\e(t)+\widetilde{K}^\e(t),u^0(t)\Big)dt,\\
 \widetilde{M}^\e(0)&=&0.
 \end{array}
\right.
\end{equation}

Since
\begin{eqnarray}\label{eq 02}
\lim_{\e\to0} \Big[\sup_{t\in[0,T]}|\widetilde{K}^\e(t)|^2+\int_0^T\|\widetilde{K}^\e(t)\|^2dt\Big]\rightarrow 0,\ \mathbb{P}^1-a.s.,
\end{eqnarray}
taking into account  (\ref{eq 00}), by standard arguments (see \cite{Temam}), we have
\begin{eqnarray}\label{eq 03}
& &\sup_{\e\in(0,\e_0]} \Big[\sup_{t\in[0,T]}|\widetilde{M}^\e(t)|^2+\int_0^T\|\widetilde{M}^\e(t)\|^2dt\Big]
+
\Big[\sup_{t\in[0,T]}|\widetilde{M}(t)|^2+\int_0^T\|\widetilde{M}(t)\|^2dt\Big]\nonumber\\
&\leq& C(\omega^1)<\infty,\ \mathbb{P}^1-a.s..
\end{eqnarray}
Set $\overline{\widetilde{M}^\e}=\widetilde{M}^\e-\widetilde{M}$ and $\overline{\widetilde{\Gamma}^\e}=\widetilde{\Gamma}^\e-\widetilde{\Gamma}$. Now the proof of (\ref{lim 01}) reduces to the proof of
\begin{equation}\label{eq 03-1}
\lim_{\e\to 0}\Big[\sup_{t\in[0,T]}|\overline{\widetilde{M}^\e}(t)|^2+\int_0^T\|\overline{\widetilde{M}^\e}(s)\|^2ds\Big]=0,\ \ \mathbb{P}^1-a.s.,
\end{equation}
We have
\begin{eqnarray}\label{eq M}
& &|\overline{\widetilde{M}^\e}(t)|^2+2\nu\int_0^t\|\overline{\widetilde{M}^\e}(s)\|^2ds\nonumber\\
&=&
-2a(\e)\int_0^t
\left (
B\Big(\widetilde{M}^\e(s)+\widetilde{\Gamma}^\e(s)+\widetilde{K}^\e(s),\widetilde{M}^\e(s)+\widetilde{\Gamma}^\e(s)+\widetilde{K}^\e(s)\Big), \overline{\widetilde{M}^\e}(s)
\right )
ds\nonumber\\
& &-2\int_0^t
\left (
B\Big(u^0(s),\overline{\widetilde{M}^\e}(s)+\overline{\widetilde{\Gamma}^\e}(s)+\widetilde{K}^\e(s)\Big), \overline{\widetilde{M}^\e}(s)
\right )
ds\nonumber\\
& &-2\int_0^t
\left (
B\Big(\overline{\widetilde{M}^\e}(s),u^0(s)\Big), \overline{\widetilde{M}^\e}(s)
\right )
ds\nonumber\\
& &-2\int_0^t
\left (
B\Big(\overline{\widetilde{\Gamma}^\e}(s)+\widetilde{K}^\e(s),u^0(s)\Big), \overline{\widetilde{M}^\e}(s)
\right )
ds\nonumber\\
&=&
I_1(t)+I_2(t)+I_3(t)+I_4(t).
\end{eqnarray}

Fix $\omega^1\in\Omega^1$. By (\ref{b4}) and (\ref{eq 03}), we have
\begin{eqnarray}\label{eq I1}
& &|I_1(t)|\nonumber\\
&\leq&
4a(\e)\int_0^t
|\widetilde{M}^\e(s)+\widetilde{\Gamma}^\e(s)+\widetilde{K}^\e(s)|
\|\widetilde{M}^\e(s)+\widetilde{\Gamma}^\e(s)+\widetilde{K}^\e(s)\|
\|\overline{\widetilde{M}^\e}(s)\|
ds\nonumber\\
&\leq&
a(\e)\int_0^t\|\overline{\widetilde{M}^\e}(s)\|^2ds\nonumber\\
&&+
2a(\e)\int_0^t
|\widetilde{M}^\e(s)+\widetilde{\Gamma}^\e(s)+\widetilde{K}^\e(s)|^2
\|\widetilde{M}^\e(s)+\widetilde{\Gamma}^\e(s)+\widetilde{K}^\e(s)\|^2
ds\nonumber\\
&\leq&
a(\e)\int_0^t\|\overline{\widetilde{M}^\e}(s)\|^2ds
+
a(\e)C(\omega^1),
\end{eqnarray}
and
\begin{eqnarray}\label{eq I2}
& &|I_2(t)|\nonumber\\
&=&2\Big|\int_0^t
\left (
B\Big(u^0(s),\overline{\widetilde{\Gamma}^\e}(s)+\widetilde{K}^\e(s)\Big), \overline{\widetilde{M}^\e}(s)
\right )
ds\Big|\nonumber\\
&\leq&
4\int_0^t
|u^0(s)|^{1/2}\|u^0(s)\|^{1/2}
|\overline{\widetilde{\Gamma}^\e}(s)+\widetilde{K}^\e(s)|^{1/2}\|\overline{\widetilde{\Gamma}^\e}(s)+\widetilde{K}^\e(s)\|^{1/2}
\|\overline{\widetilde{M}^\e}(s)\|
ds\nonumber\\
&\leq&
\frac{1}{2}\nu \int_0^t\|\overline{\widetilde{M}^\e}(s)\|^2ds
+
C\int_0^t
|u^0(s)|\|u^0(s)\|
|\overline{\widetilde{\Gamma}^\e}(s)+\widetilde{K}^\e(s)|\|\overline{\widetilde{\Gamma}^\e}(s)+\widetilde{K}^\e(s)\|
ds\nonumber\\
&\leq&
\frac{1}{4}\nu \int_0^t\|\overline{\widetilde{M}^\e}(s)\|^2ds
+C(\omega^1)\Big[\int_0^T\|\overline{\widetilde{\Gamma}^\e}(s)+\widetilde{K}^\e(s)\|^2ds\Big]^{1/2},
\end{eqnarray}
similar to (\ref{eq I2}),
\begin{eqnarray}\label{eq I4}
|I_4(t)|
\leq
\frac{1}{4}\nu \int_0^t\|\overline{\widetilde{M}^\e}(s)\|^2ds
+C(\omega^1)\Big[\int_0^T\|\overline{\widetilde{\Gamma}^\e}(s)+\widetilde{K}^\e(s)\|^2ds\Big]^{1/2}.
\end{eqnarray}

By (\ref{b6}), (\ref{L4}) and (\ref{eq 03}),
\begin{eqnarray}\label{eq I3}
|I_3(t)|
&=&
2\Big|\int_0^t
\left (
B\Big(\overline{\widetilde{M}^\e}(s),\overline{\widetilde{M}^\e}(s)\Big), u^0(s)
\right )
ds\Big|\nonumber\\
&\leq&
\nu \int_0^t\|\overline{\widetilde{M}^\e}(s)\|^2ds
+
C\int_0^t\|u^0(s)\|^2|\overline{\widetilde{M}^\e}(s)|^2ds.
\end{eqnarray}
Combining (\ref{eq M})--(\ref{eq I3}), we have
\begin{eqnarray*}
& &|\overline{\widetilde{M}^\e}(t)|^2+\Big(1/2-a(\e)\Big)\nu\int_0^t\|\overline{\widetilde{M}^\e}(s)\|^2ds\nonumber\\
&\leq&
a(\e)C(\omega^1)+C(\omega^1)\Big[\int_0^T\|\overline{\widetilde{\Gamma}^\e}(s)+\widetilde{K}^\e(s)\|^2ds\Big]^{1/2}\nonumber\\
& &+
C\int_0^t\|u^0(s)\|^2|\overline{\widetilde{M}^\e}(s)|^2ds.
\end{eqnarray*}
Since $\lim_{\e\to 0}a(\e)=0$ and
$$
\lim_{\e\to 0}\Big[\int_0^T\|\overline{\widetilde{\Gamma}^\e}(s)+\widetilde{K}^\e(s)\|^2ds\Big]=0,\ \ \mathbb{P}^1-a.s.,
$$
by Gronwall's lemma we obtain
$$
\lim_{\e\to 0}\Big[\sup_{t\in[0,T]}|\overline{\widetilde{M}^\e}(t)|^2+\int_0^T\|\overline{\widetilde{M}^\e}(s)\|^2ds\Big]=0,\ \ \mathbb{P}^1-a.s.
$$
The proof is complete.
\end{proof}

\def\refname{ References}

\end{document}